\documentclass[a4paper,12pt]{article}
\usepackage{mathrsfs}
\usepackage{}
\usepackage[top=2.5cm,bottom=2.5cm,left=2.5cm,right=2.5cm]{geometry}
\usepackage{amssymb}
\usepackage{graphicx}
\usepackage{amsmath,amsthm,amssymb,lineno}
\usepackage{latexsym}
\usepackage{epstopdf}
\usepackage{setspace}
\usepackage{graphicx,booktabs,multirow}
\usepackage{latexsym, tabularx,shapepar}
\usepackage[all,2cell,dvips]{xy} \UseAllTwocells \SilentMatrices
\usepackage{appendix}
\usepackage{longtable}
\usepackage{cite}
\usepackage{CJK}
\usepackage{indentfirst}
\usepackage{array}
\usepackage{amsmath}
\allowdisplaybreaks[4]
\graphicspath{{figures/}}

\newtheorem{theorem}{Theorem}[section]
\newtheorem{lemma}[theorem]{Lemma}
\newtheorem{corollary}[theorem]{\rm\bfseries Corollary}

\begin{document}

\title{ Unicyclic signed graphs with maximal energy
\thanks{
}}
\author{Dijian Wang, Yaoping Hou \thanks{Corresponding author: yphou@hunnu.edu.cn}\\
{\footnotesize Key Laboratory of High Performance Computing and Stochastic Information} \\
{\footnotesize Processing (HPCSIP) (Ministry of Education of China)} \\
{\footnotesize College of Mathematics and Statistics, Hunan Normal University} \\
{\footnotesize Changsha, Hunan 410081, P. R. China.} \\
}

\date{}
\maketitle
\begin{abstract}
Let $x_1, x_2, \dots, x_n$ be the eigenvalues of a signed graph $\Gamma$ of order $n$.
The energy of $\Gamma$ is defined as $E(\Gamma)=\sum^{n}_{j=1}|x_j|.$ Let $\mathcal{P}_n^4$ be obtained by connecting a vertex of
the negative circle $(C_4,{\overline{\sigma}})$ with a terminal vertex of the path $P_{n-4}$. In this paper, we show that for $n=4,6$ and $n \geq 8,$ $\mathcal{P}_n^4$ has the maximal energy
among all connected unicyclic $n$-vertex signed graphs, except the cycles $C_5^+, C_7^+.$ 
\end{abstract}

\baselineskip=0.30in

\section {Introduction}
In this paper, all graphs are simple (loopless and without
multiple edges). The vertex set and edge set of the graph $G$ will be denoted by
$V(G)$ and $E(G)$, respectively. A \emph{signed graph}  $\Gamma = (G, \sigma)$ (or $\Gamma =G^\sigma$)
consists of an unsigned graph $G = (V, E)$ and a sign function $\sigma: E(G)\rightarrow  \{+,-\}$, and
$G$ is its underlying graph, while $\sigma$ is its sign
function (or signature).  An edge $e$ is \emph{positive} (\emph{negative}) if $\sigma(e) = +$ (resp. $\sigma(e) = -$).
If all edges in $\Gamma$ are \emph{positive} (\emph{negative}), then $\Gamma$ is denoted by $(G,+)$ (resp. $(G,-)$).

Actually, each concept defined for the underlying graph can be transferred with  the signed graph.
For example, the degree of a vertex $v$ in $G$ is also its degree in $\Gamma$.
Furthermore, if some subgraph of the underlying graph is observed, then the sign function for the signed subgraph is the
restriction of the previous one. Thus, if $v \in V(G)$, then $\Gamma -v$ denotes the signed subgraph
having $G - v$ as the underlying graph, while its signature is the restriction from $E(G)$ to
$E(G - v)$ (note, all edges incident to $v$ are deleted). If $U \subset V(G)$ then $\Gamma[U]$ or $G(U)$
denotes the (signed) induced subgraph arising from $U$, while $\Gamma- U = \Gamma[V(G)\backslash U]$.
Sometimes we also write $\Gamma-\Gamma[U]$ instead of $\Gamma- U$.
 Let $C$ be a cycle in $\Gamma$, the sign of $C$ is given by $\sigma(C) =\prod_{e\in C}\sigma(e).$  A cycle whose sign is + (resp. $-$) is called \emph{positive} (resp. \emph{negative}), denotes by $C^+$ (resp. $C^{\overline{\sigma}}$). Alternatively, we can say that a cycle is positive if it contains an even number of negative edges. A signed graph is \emph{balanced} if all cycles are positive; otherwise it is \emph{unbalanced}. \emph{Switching} $\Gamma$
 means reversing the signs of all edges between a vertex subset $U$ and its complement. $U$ may be empty. The switched signed graphs is written $\Gamma^U$, and we call
$\Gamma^U$ and $\Gamma$ are \emph{switching equivalent}. Switching equivalence leaves the many  graphic
invariants, such as the set of positive cycles. In fact, the signature on bridges is not relevant,
 hence we will always
consider the all positive signature for trees. In the sequel signed trees and unsigned trees
will be considered as the same object. For the same reason, the edges which do not lie on
any cycle are not relevant for the signature and they will be always considered as positive.

The adjacency matrix of a signed graphs $\Gamma$ whose vertices are $v_1, v_2,\dots , v_n$ is the $n \times n$ matrix $A(\Gamma)=(a_{ij}),$ where
$$ a_{ij}=\left\{
\begin{array}{rcl}
\sigma(v_iv_j),    &      & \text{if there is an edge between $v_i$ and $v_j$, }\\
0    , &      & \text{otherwise.}\\
\end{array} \right.$$

Clearly, $A(\Gamma)$ is real symmetric and so all its eigenvalues are real. The characteristic polynomial $|xI-A(\Gamma)|$ of the adjacency
matrix $A(\Gamma)$ of a signed graph $\Gamma$ is called the characteristic polynomial of $\Gamma$ and is denoted by $\phi(x).$ The eigenvalues of $A(\Gamma)$ are called the eigenvalues of $\Gamma$. The set of eigenvalues of $\Gamma$ together with their multiplicities are called the spectrum of $\Gamma$. If
$\Gamma$ is a signed graph of order $n$ having distinct eigenvalues $x_1, x_2, \dots, x_k$ and their respective multiplicities as $m_1,m_2, \dots ,m_k,$
we write the spectrum of $\Gamma$ as $spec(\Gamma) = \{x^{m_1}, x^{m_2}, \dots , x^{m_k} \}.$

The following is the coefficient theorem for signed graphs\cite{A80}.

\begin{theorem}\label{thm1.1}
  If $\Gamma$ is a signed graph with characteristic polynomial
  $$\phi(x)=x^n + a_1(\Gamma)x^{n-1} + \cdot\cdot\cdot + a_{n-1}(\Gamma)x + a_n(\Gamma),$$
then
$$a_j(\Gamma)=\sum_{L\in \mathcal{L}_j}(-1)^{P(L)}2^{|c(L)|}\prod_{Z\in c(L)}\sigma(Z).$$
for all $j = 1, 2,\dots , n,$ where $\mathcal{L}_j$ is the set of all basic figures $L$ of $\Gamma$ of order $j$, $p(L)$ denotes number of components of $L,$  $c(L)$
denotes the set of all cycles of $L$ and $\sigma(Z)$ is the sign of cycle $Z.$
\end{theorem}
From Theorem \ref{thm1.1}, it is clear that the spectrum of a signed graphs remains invariant by changing the signs of non-cyclic
edges. The spectral criterion for the balance of signed  graphs given by Acharya \cite{A80} is as follows.

\begin{theorem}\label{thm1.2}
  A signed graph is balanced if and only if it is co-spectral with the underlying unsigned graph.
\end{theorem}

The concept of energy of a graph was given by Gutman \cite{G78} in 1978. This concept was extended to signed  graphs by Germina, Hameed
and Zaslavsky \cite{G10} and they defined the energy of a signed graphs $\Gamma$ to be the sum of absolute values of eigenvalues of $\Gamma$. The
concept of energy has been extended to digraphs by Pena and Rada \cite{P08} and to signed digraphs by Bhat \cite{P14}. For
applications of signed  graphs in chemistry see \cite{G95}. For unsigned graphs and signed graphs with extremal energy
see \cite{H02,H11,HB11,M17,L08,G07,GY01}.

The girth of a graph (signed graph) is the length of its smallest cycle and is denoted by $g$. Let $P^g_n$ ($\mathcal{P}^g_n$) denotes the
balanced (unbalanced) unicyclic signed graph of order $n$ obtained by connecting a vertex of
 $C_g^+$ ($C_g^{\overline{\sigma}}$) with a terminal vertex of the path $P_{n-g},$ where $n\geq g\geq  3$
and let $\mathbb{U}(n, g)$ denotes the set of unicyclic signed  graphs of order $n$ and girth $g$.
$P^g_n$ and $\mathcal{P}^4_n$ are depicted in Fig \ref{Fig-1}, positive edges are depicted as
bold lines, negative edges as dash lines. For any unbalanced unicyclic signed graph, it is switching equivalent to the unicyclic signed graph such that exactly one negative edge in the circle, and the rest edges are positive. Then all unbalanced signed graphs $(P^4_n,\sigma)$ are switching equivalent to $\mathcal{P}^4_n.$

\begin{figure}
\begin{center}
  \includegraphics[width=10cm,height=2cm]{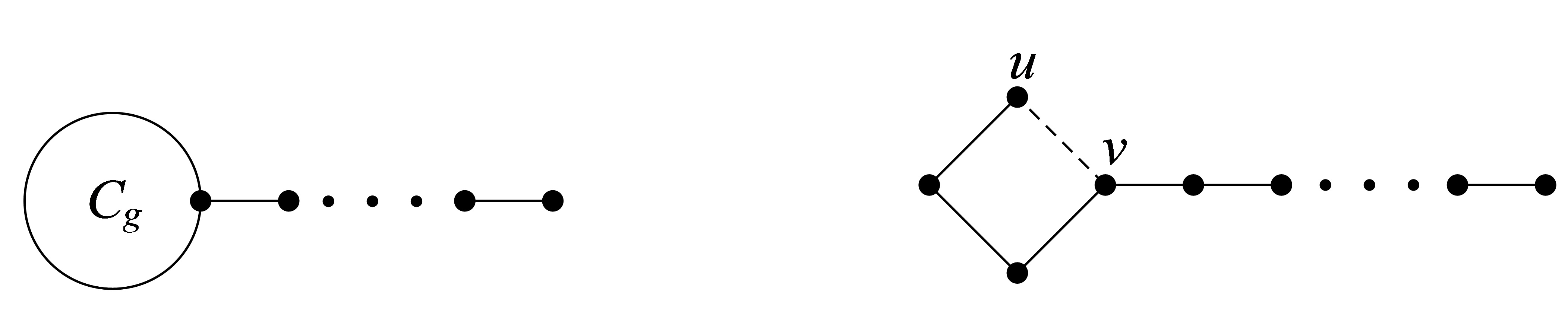}
  \end{center}
   \vskip -0.8cm\caption{Graph $P_n^g$ and signed graphs $\mathcal{P}_n^4$.}
  \label{Fig-1}
\end{figure}

The organization of this paper is following:
 in Section 2, we give some well known results which will be used in this paper. In section 3, we prove that for $n=4,6$ and $n \geq 8,$ $\mathcal{P}_n^4$ has the maximal energy among all connected unicyclic $n$-vertex signed graphs, except the cycles $C_5^+, C_7^+.$

\section{Preliminaries}

Firstly, Gill and Acharya \cite{GA80} obtained the following recurrence formula for the characteristic polynomial of a signed graph.

\begin{lemma}\label{lem2.1}
Let $uv$ be an edge of a signed graph $\Gamma,$ then
  $$\phi(\Gamma,x)=\phi(\Gamma-uv,x)-\phi(\Gamma-u-v,x)-2\sum_{C\in \mathcal{C}_{uv}}sgn(C)\phi(G-C,x),$$
  where $\mathcal{C}_{uv}$ is the set of cycles containing $uv$.

  In particular, if $uv$ is a pendant edge of $\Gamma$ with the pendant vertex $v$, then
  $$\phi(\Gamma,x)=x\phi(\Gamma-v,x)-\phi(\Gamma-u-v,x).$$
\end{lemma}

The energy of a signed graph $\Gamma,$ denoted by $E(\Gamma),$ is defined by Germina, Hameed and
Zaslavsky \cite{G10} as following: $E(\Gamma)=\sum^{n}_{j=1}|x_j|,$ where $x_1, x_2, \dots , x_n$ are the eigenvalues of signed graphs $\Gamma$.
 The \emph{Coulson integral formula} \cite{M17,G2001} is
$$E(\Gamma)=\frac{1}{\pi}\int_{-\infty}^{\infty}\frac{1}{x^2}\log\bigl|x^{n}\phi(\frac{i}{x})\bigl|dx,$$
where $i=\sqrt{-1}$ and
$\int_{-\infty}^{\infty}F(x)dx$ denotes the principle value of the respective integral. Moreover, it is known that the above equality can be expressed
an explicit formula as follows:
$$E(\Gamma)=\frac{1}{2\pi}\int_{-\infty}^{\infty}\frac{1}{x^2}\log\bigg[\bigg(\sum_{j=0}^{\lfloor\frac{n}{2}\rfloor}(-1)^ja_{2j}(\Gamma)x^{2j}\bigg)^2+    \bigg(\sum_{j=0}^{\lfloor\frac{n}{2}\rfloor}(-1)^ja_{2j+1}(\Gamma)x^{2j+1}\bigg)^2\bigg]dx.$$
where $a_1, a_2,\dots , a_n$ are the coefficients of the characteristic polynomial $\phi(x).$

We know that a graph containing at least one edge is bipartite if and only if its spectrum,
considered as a set of points on the real axis, is symmetric with respect to the origin. This is not true for signed  graphs.
There exist non bipartite signed  graphs whose spectrum is symmetric about the origin. Two counter-examples are given by Bhat \cite{M17}.
We say that a signed graphs has the pairing property if its spectrum is symmetric with respect to origin. We denote by $\Delta_n,$ the set of all signed  graphs on $n$ vertices with pairing property. Put $b_{2j}(\Gamma) = (-1)^ja_{2j}(\Gamma),$ $b_{2j+1}(\Gamma) = (-1)^ja_{2j+1}(\Gamma).$ Note $b_1(\Gamma) = 0,$ $b_2(\Gamma)$ = the number of edges of the
signed graphs $\Gamma$ and so on.

The next result shows that all odd coefficients of a signed graphs in $\Delta_n,$ are zero and all even coefficients alternate in
sign \cite{B2015}.

\begin{lemma}\label{lem2.2}\cite{B2015}
Let $\Gamma$ be a signed graph of order $n.$ Then the following statements are equivalent:\\
(i) Spectrum of $\Gamma$ is symmetric about the origin.\\
(ii) $\phi(x)=x^n + \sum^{\lfloor\frac{n}{2} \rfloor}_{k=1}(-1)^kb_{2k}(\Gamma)x^{n-2k},$ where $b_{2k}(\Gamma)=|a_{2k}(S)|$ for all $k = 1, 2,\dots , \lfloor\frac{n}{2} \rfloor.$\\
(iii) $\Gamma$ and $-\Gamma$ are co-spectral, where $-\Gamma$ is the signed graph obtained by negating sign of each edge of $\Gamma.$
\end{lemma}

 Obviously, if $\Gamma$ is a bipartite signed graph, then $\Gamma$ has the pairing property, i.e., $\Gamma\in \Delta_n.$

Now, we define a quasi-order relation for signed  graphs in $\Gamma\in \Delta_n$ and show that it is possible to compare the energies of signed  graphs in $\Gamma\in \Delta_n.$

Given signed  graphs $\Gamma_1$ and $\Gamma_2$ in $\Delta_n$, by Lemma \ref{lem2.2}, for $i = 1, 2,$ we have
$$\phi_i(x)=x^n + \sum^{\lfloor\frac{n}{2} \rfloor}_{k=1}(-1)^kb_{2k}(\Gamma_i)x^{n-2k},$$
where $b_{2j}(\Gamma_i)$ are non negative integers for all $j = 1, 2,\dots , \lfloor\frac{n}{2} \rfloor.$ If $b_{2j}(\Gamma_1) \leq b_{2j}(\Gamma_2)$
for all $j = 1, 2,\dots , \lfloor\frac{n}{2} \rfloor,$ we define
$\Gamma_1 \preceq \Gamma_2.$  If in addition $b_{2j}(\Gamma_1)< b_{2j}(\Gamma_2)$ for some $j = 1, 2,\dots , \lfloor\frac{n}{2} \rfloor,$ we write $\Gamma_1 \prec \Gamma_2.$ Clearly, $\preceq$ is a quasi-order relation.
The following result shows that the energy increases with respect to this quasi-order
relation.

\begin{lemma}\label{lem2.3}\cite{M17}
If $\Gamma\in \Delta_n,$ then
  $$E(\Gamma)=\frac{1}{\pi}\int_{-\infty}^{\infty}\frac{1}{x^2}\log\bigl[1+\sum_{j=1}^{\lfloor\frac{n}{2}\rfloor}b_{2j}(\Gamma)x^{2j}\bigl]dx.$$
  In particular, if $\Gamma_1,\Gamma_2\in \Delta_n$ and $\Gamma_1\prec\Gamma_2,$ then $E(\Gamma_1) < E(\Gamma_2).$
\end{lemma}

For the $k$-matching number of a graph $G,$ we have the following \cite{C79}.

\begin{lemma}\label{lem2.4}
  Let $e = uv$ be an edge of $G.$ Then

(i) $m(G, k) = m(G -e, k) + m(G-u-v, k-1).$

(ii) If $G$ is a forest, then $m(G, k)\leq m(P_n, k), k \geq 1.$

(iii) If $H$ is a subgraph of G, then $m(H, k) \leq m(G, k), k \geq 1.$ Moreover, if H is
a proper subgraph of G, then the inequality is strict for some $k$.
\end{lemma}

Since the matching number is independent of signature, then Lemma \ref{lem2.4} also holds for signed graphs and $m(\Gamma,i)=m(G,i)$. We denote by $m(\Gamma, i)$ the number of matchings of $\Gamma$ of size $i.$  For convention, we let  $m(\Gamma, 0)=1$ and $m(\Gamma, i)=0$ for $i\geq\frac{n}{2}.$

\begin{lemma}\label{lem2.5}\cite{G86}
  Let $n = 4k, 4k + 1, 4k + 2$ or $4k + 3.$ Then
\begin{align*}
  P_n &\succ P_2\cup P_{n-2}\succ P_4\cup P_{n-4}\succ\dots \succ P_{2k}\cup P_{n-2k}\succ P_{2k+1}\cup P_{n-2k-1}\\
  & \succ P_{2k-1}\cup P_{n-2k+1}\succ \dots \succ P_3\cup P_{n-3}\succ P_1\cup P_{n-1}.
\end{align*}

\end{lemma}

\section{Unicyclic signed graphs with maximal energy.}

In this section, we will prove that for $n=4,6$ and $n \geq 8,$ $\mathcal{P}_n^4$ has the maximal energy
among all connected unicyclic  signed  graphs on $n$ vertices, except the cycles $C_5^+, C_7^+.$

In \cite{C99}, Caporossi et al. proposed the following conjecture
on the unicyclic graph with maximal energy.

\noindent \textbf{Conjecture 1.} \emph{Among all unicyclic graphs on $n$ vertices, the cycle $C_n$ has maximal energy if $n \leq 7$ and
$n = 9, 10, 11, 13$ and $15.$ For all other values of $n,$ the unicyclic graph with maximal energy is $P^6_n.$}

The following result of Hou and Huo \cite{H02,H11,HB11} proves the conjecture.

\begin{theorem}\label{lem3.1}
Among all unicyclic graphs on $n$ vertices, the cycle $C_n$ has maximal energy if $n \leq 7$ and
$n = 9, 10, 11, 13$ and $15.$ For all other values of $n$, the unicyclic graph with maximal energy is $P^6_n.$
\end{theorem}

In \cite{M17}, Bhat and Pirzada gave some results about the energy of the unicyclic signed  graphs as following.

\begin{lemma}\label{lem3.2}\cite{M17}
Let G be a unicyclic graph of odd girth. Then any two signed graphs on $G$ have the same energy.
\end{lemma}

\begin{lemma}\label{lem3.3}\cite{M17}
Let $G$ be a unicyclic graph of order $n$ and even girth $g,$ i.e., the bipartite unicyclic graph and let S be any balanced
signed graph on $G$ and $T$ be any unbalanced one. Then

(i) $E(S)<E(T)$ if and only if $g$ $\equiv$ $0$ $($mod  $4)$,

(ii)$E(S)>E(T)$ if and only if $g$ $\equiv$ $2$ $($mod  $4)$.
\end{lemma}

By Theorem \ref{lem3.1}, Lemmas \ref{lem3.2} and \ref{lem3.3}, we can easily obtain the following corollary.

\begin{corollary}\label{cor3.4}
  Let $G$ be a unicyclic graph with girth $g$ is odd or $g$ $\equiv$ $2$ $($mod  $4)$. For any signed graph $\Gamma=(G,\sigma)\in \mathbb{U}(n,g),$ we have $E(\Gamma)\leq E(C_n)$ for $n \leq 7,$
$n = 9, 10, 11, 13,15,$ and $E(\Gamma)\leq E(P^6_n)$ for all other values of $n.$
\end{corollary}

Next we consider the unicyclic unbalanced signed  graphs with girth $g$ $\equiv$ $0$ $($mod  $4).$
By Lemma \ref{lem2.1}, we have the following observation.

\begin{lemma}\label{lem3.5}
  Let $\Gamma\in \mathbb{U}(n,g)$ be unbalanced and let $uv$ be the pendant edge of $\Gamma$ with the pendant vertex v. Then
  $$b_j(\Gamma) = b_j(\Gamma - v) + b_{j-2}(\Gamma - v - u).$$

\end{lemma}

The following result shows that the energy of the (signed) trees is less than $E(\mathcal{P}_n^4)$.

\begin{lemma}\label{lem3.6}
Let $T$ be any  \emph{(}signed\emph{)} tree with $n\geq 12$ vertices.   Then $T\prec \mathcal{P}_n^4.$
\end{lemma}
\begin{proof}
By Lemma \ref{lem2.4}, we can get $T\prec P_n.$ Choosing an edge $uv\in\mathcal{P}_n^4$ as depicted in Figure \ref{Fig-1}, by Lemma \ref{lem2.1}, we have
\begin{align*}
  b_{2i}(\mathcal{P}_{n}^4) &= (-1)^ia_{2i}(\mathcal{P}_{n}^4)\\ &=(-1)^ia_{2i}(\mathcal{P}_{n}^4-uv)-(-1)^ia_{2i-2}(\mathcal{P}_{n}^4-u-v)+(-1)^ia_{2i-4}(\mathcal{P}_{n}^4-C_4^{\overline{\sigma}}) \\
  & =b_{2i}(P_n)+b_{2i-2}(P_{n-4}\cup P_2)+b_{2i-4}(P_{n-4})\\
  & >b_{2i}(P_n).
\end{align*}
By Lemma \ref{lem2.3}, we have $T\prec P_n\prec \mathcal{P}_{n}^4.$
\end{proof}

The following result shows that among all unbalanced unicyclic signed  graphs with girth $g$ $\equiv$ $0$ $($mod  $4)$ in $\mathbb{U}(n, g),$ $\mathcal{P}_n^g$ has maximal energy.

\begin{lemma}\label{lem3.7}
Let $\Gamma\in \mathbb{U}(n, g)$ be unbalanced and  $g$ $\equiv$ $0$ $($mod  $4).$ If $\Gamma \nsim \mathcal{P}_n^g,$ then $\Gamma \prec \mathcal{P}_n^g.$
\end{lemma}
\begin{proof}
 We prove the Lemma \ref{lem3.7} by induction on $n-g$.

In a trivial manner the Lemma \ref{lem3.7} holds for $n-g= 0$ and $n-g=1,$
because then $\mathbb{U}(n, g)$ has only a single element. Let $p \geq 2$ and suppose the result
is true for $n-g<p.$ Now we consider $n-g=p.$ Since $\Gamma$ is unicyclic and not a
cycle, for $n > p$, $\Gamma$ must have a pendant edge $uv$ with pendant vertex $v.$ As $\Gamma\nsim \mathcal{P}_n^g $ and $n\geq g + 2,$ we may choose a pendant edge $uv$ of $\Gamma$ such that $\Gamma-v\nsim \mathcal{P}_{n-1}^g. $ By Lemma \ref{lem3.5}, we have
\begin{equation}\label{eq1}
b_i(\Gamma)=b_{i}(\Gamma-v)+b_{i-2}(\Gamma-v-u),
\end{equation}
\begin{equation}\label{eq2}
b_i(\mathcal{P}_n^g)=b_{i}(\mathcal{P}_{n-1}^g)+b_{i-2}(\mathcal{P}_{n-2}^g).
\end{equation}
By the induction assumption, we have $\Gamma-v\prec \mathcal{P}_{n-1}^g. $

If $\Gamma-v-u$ contains the cycle $(C_g,\overline{\sigma})$, then by the induction hypoyhesis, we have
$\Gamma-v-u\prec \mathcal{P}_{n-2}^g.$ (It is easy to show that this relation holds also if $\Gamma-v-u$
is not connected.) Thus Lemma \ref{lem3.7} follows from Eqs. (\ref{eq1}), (\ref{eq2})  and the inductive
hypothesis.

If $\Gamma-v-u$ does not contain the cycle $(C_g,\overline{\sigma})$, then it is acyclic. Then $b_{i-2}(\Gamma-v-u) = 0$ when $i$ is odd whereas for $i = 2k,$
by Lemma \ref{lem3.6}, we have $\Gamma-v-u\preceq P_{n-2}$ and

\begin{equation}\label{eq3}
\begin{split}
  b_{2i}(\mathcal{P}_{n-2}^g) &= (-1)^ia_{2i}(\mathcal{P}_{n-2}^g)\\ &=(-1)^ia_{2i}(\mathcal{P}_{n-2}^g-uv)-(-1)^ia_{2i-2}(\mathcal{P}_{n-2}^g-u-v)+(-1)^ia_{2i-g}(\mathcal{P}_{n-2}^g-C_g^{\overline{\sigma}}) \\
  & =b_{2i}(P_{n-2})+b_{2i-2}(P_{n-g-2}\cup P_{g-2})+b_{2i-g}(P_{n-g-2})\\
  & >b_{2i}(P_{n-2}).
\end{split}
\end{equation}
Then $\Gamma-v-u\preceq P_{n-2}\prec \mathcal{P}_{n-2}^g$ and
Lemma \ref{lem3.7} now follows from Eqs. (\ref{eq1}), (\ref{eq2}), (\ref{eq3}) and the inductive assumption.
\end{proof}

\begin{lemma}\label{lem3.8}
  Let $n\geq 4$ and n $\equiv$ $0$ $($mod  $4).$ Then $(C_n,{\overline{\sigma}})\prec \mathcal{P}_{n}^4.$
\end{lemma}
\begin{proof}
  By Theorem \ref{thm1.1}, we have $b_n(C_n,{\overline{\sigma}})=b_n(\mathcal{P}_{n}^4)=4.$

Since both of $(C_n,{\overline{\sigma}})$ and $\mathcal{P}_{n}^4$  are unbalanced unicyclic signed  graphs, we choose one negative edge $uv$ in the cycle, and the rest is positive.  By Lemma \ref{lem2.1}, for $2i<n$, we have
\begin{align*}
  b_{2i}(C_n,{\overline{\sigma}}) & =(-1)^ia_{2i}(C_n,{\overline{\sigma}})=(-1)^ia_{2i}((C_n,{\overline{\sigma}})-uv)-(-1)^ia_{2i-2}((C_n,{\overline{\sigma}})-u-v) \\
  & =b_{2i}(P_n)+b_{2i-2}(P_{n-2})\\
  & =b_{2i}(P_n)+b_{2i-2}(P_{n-4}\cup P_2)+b_{2i-4}(P_{n-5}\cup K_1)\\
 & <b_{2i}(P_n)+b_{2i-2}(P_{n-4}\cup P_2)+b_{2i-4}(P_{n-4}) ~~~~ \mbox{(by Lemma  \ref{lem2.5})}\\
  & =  b_{2i}(\mathcal{P}_{n}^4),
\end{align*}
Then  $b_{2i}(C_n,{\overline{\sigma}})\leq b_{2i}(\mathcal{P}_{n}^4),$ for $i=1,2,\dots,\lfloor\frac{n}{2} \rfloor.$ By  Lemma \ref{lem2.3}, we have $(C_n,{\overline{\sigma}})\prec \mathcal{P}_{n}^4.$
\end{proof}

\begin{lemma}\label{lem3.9}
Let $g\geq8$ and g $\equiv$ $0$ $($mod  $4).$ Then $\mathcal{P}_{g+1}^g \prec \mathcal{P}_{g+1}^4$ and $\mathcal{P}_{g+2}^g \prec \mathcal{P}_{g+2}^4.$
\end{lemma}
\begin{proof} Choosing an negative edge $uv\in C_g^{\overline{\sigma}}$ such that $u$ is the vertex of degree 3 of $C_g^{\overline{\sigma}}$. Deleting it, and applying Lemma \ref{lem2.1}, we have
\begin{align*}
  b_{2i}(\mathcal{P}_{g+1}^4) &= b_{2i}(P_{g+1})+b_{2i-2}(P_{g-3}\cup P_2)+b_{2i-4}(P_{g-3}). \\
b_{2i}(\mathcal{P}_{g+1}^g) &= b_{2i}(P_{g+1})+b_{2i-2}(P_{g-2}\cup K_1)+b_{2i-g}(K_{1}). \\
 b_{2i}(\mathcal{P}_{g+2}^4) &= b_{2i}(P_{g+2})+b_{2i-2}(P_{g-2}\cup P_2)+b_{2i-4}(P_{g-2}). \\
 b_{2i}(\mathcal{P}_{g+2}^g) &=b_{2i}(P_{g+2})+b_{2i-2}(P_{g-2}\cup P_2)+b_{2i-g}(P_2).
\end{align*}

By Lemma \ref{lem2.5}, we have $P_{g-2}\cup K_1\prec P_{g-3}\cup P_2$ and $P_2 \prec P_{g-2}.$ Then
$b_{2i}(\mathcal{P}_{g+1}^4)> b_{2i}(\mathcal{P}_{g+1}^g)$ and $ b_{2i}(\mathcal{P}_{g+2}^4)>  b_{2i}(\mathcal{P}_{g+2}^g).$
\end{proof}

\begin{lemma}\label{lem3.10}
Let $g\geq8$ and g $\equiv$ $0$ $($mod  $4).$ Then $\mathcal{P}_{n}^g \prec \mathcal{P}_{n}^4.$
\end{lemma}
\begin{proof} We prove the result by induction on $n$. By Lemma \ref{lem3.9} the statement
is true for $n = g + 1$ and $n = g + 2$. Suppose that $n > g + 2$ and that Lemma \ref{lem3.10}
holds for $n -1$ and $n- 2$. By Lemma \ref{lem3.5}, we have
\begin{align*}
  b_{2k}(\mathcal{P}_{n}^g) &= b_{2k}(\mathcal{P}_{n-1}^g)+b_{2k-2}(\mathcal{P}_{n-2}^g),\\
 b_{2k}(\mathcal{P}_{n}^4) &= b_{2k}(\mathcal{P}_{n-1}^4)+b_{2k-2}(\mathcal{P}_{n-2}^4)
\end{align*}
\end{proof}

Next we determine unicyclic signed  graphs with maximal energy.

\begin{lemma}\label{lem3.11}
Let $n\geq 6$. Then  $P_n^6 \prec \mathcal{P}_{n}^4.$
\end{lemma}
\begin{proof} We verify Lemma \ref{lem3.11} by induction on $n$. By direct calculation we check that
the Lemma \ref{lem3.11} holds for $n = 6$ and $n = 7$. Indeed, the respective characteristic polynomials are following:
$$   \phi(P_6^6, x)=x^6-6x^4 + 9x^2-4, \quad  \phi(P_7^6, x)=x^7-7x^5 + 13x^3-7x, $$
$$  \phi(\mathcal{P}_6^4, x)=x^6-6x^4 + 10x^2-4, \quad  \phi(\mathcal{P}_7^4, x)=x^7+7x^5 -15x^3-10x.$$

Now suppose that $n\geq 8$ and that the statement of the Lemma \ref{lem3.11} is true for the signed  graphs
with $n-1$ and $n-2$ vertices, i.e., that $b_{2k}(P_{n-1}^6)\leq b_{2k}(\mathcal{P}_{n-1}^4)$,
$b_{2k-2}(P_{n-2}^6)\leq b_{2k-2}(\mathcal{P}_{n-2}^4).$ By Lemma \ref{lem3.5}, we have
$$b_{2k}({P}_{n}^6)=b_{2k}(P_{n-1}^6)+b_{2k-2}(P_{n-2}^6), $$
$$b_{2k}(\mathcal{P}_{n}^4)=b_{2k}(\mathcal{P}_{n-1}^4)+b_{2k-2}(\mathcal{P}_{n-2}^4). $$
 Hence Lemma \ref{lem3.11} follows.
\end{proof}

From the above proofs and Theorem \ref{lem3.1}, then we just need to compare the energy of $C_n^+$ and $\mathcal{P}_{n}^4$ for $n \leq 7$ and
$n = 9, 10, 11, 13, 15.$ By direct calculation,  for $n=4$ we can find  $\mathcal{P}_{4}^4$ ($C_4^{\overline{\sigma}}$)  has   maximal energy,  $E(C_n^+)>E(\mathcal{P}_{n}^4)$ for $n=5,7,$ and $E(C_n^+)<E(\mathcal{P}_{n}^4)$ for $n=6, 9, 10, 11, 13, 15.$
Thus we have
\begin{theorem}\label{lem3.12}
Among all unicyclic signed graphs on $n\geq 4$ vertices, the cycle $C_n^+$ has maximal energy if
$n = 5,7.$ For all other values of $n\geq4$, the unicyclic signed graph with maximal energy is $\mathcal{P}_{n}^4.$
\end{theorem}


\begin{thebibliography}{99}
\bibitem{A80}  B. D. Acharya, Spectral criterion for the cycle balance in networks, J. Graph Theory 4 (1980) 1--11.

\bibitem{M17}  M. A. Bhat, S. Pirzada, Unicyclic signed graphs with minimal energy, Discrete Appl. Math., Volume 226 (2017) 32--39.

\bibitem{B2015}  M. A. Bhat, S. Pirzada, On equienergetic signed graphs, Discrete Appl. Math. 189 (2015) 1--7.

\bibitem{C99} G. Caporossi, D. Cvetkovic, I. Gutman, P. Hansen, Variable neighborhood search for extremal graphs. 2. Finding graphs with
extremal energy, J. Chem. Inf. Comput. Sci. 39 (1999) 984--996.

\bibitem{C79}D. Cvetkovi\'{c}, M. Doob, and H. Sachs. Spectra of Graphs. Academic Press, New York, 1979.

\bibitem{GA80}M. K. Gill, B.D. Acharya, A recurrence formula for computing the characteristic polynomial of a signed graph, J. Comb. Inf. Syst. Sci. 5 (1980) 68--72.

\bibitem{G10} K. A. Germina, S. Hameed K, T. Zaslavsky, On products and line graphs of signed
graphs, their eigenvalues and energy. Linear Algebra Appl. 435 (2011) 2432--2450.

\bibitem{H02} Y. Hou, I. Gutman, C. Woo, Unicyclic graphs with maximal energy, Linear Algebra Appl. 356 (2002) 27--36.

\bibitem{H11} B. Huo, X. Li, Y. Shi, Complete solution to a problem on the maximal energy of
unicyclic bipartite graphs, Linear Algebra Appl. 434 (2011) 1370--1377.

\bibitem{HB11} B. Huo, X. Li, Y. Shi, Complete solution to a conjecture on the maximal energy
of unicyclic graphs, European J. Combin. 32 (2011) 662--673.

 \bibitem{L08} S. Li, X. Li, Z. Zhu, On tricyclic graphs with minimal energy, MATCH Commun. Math. Comput. Chem. 59
(2008) 397--419.

\bibitem{G2001} I. Gutman, The energy of a graph: Old and new results, in: A. Betten, A. Kohnert, R. Laue, A. Wasserman (Eds.), Algebriac Combinatorics and Applications, Springer-Verlag, Berlin, 2001, pp. 196--211.

\bibitem{G86} I. Gutman, O. E. Polansky, Mathematical Concepts in Organic Chemistry, Springer, Berlin, 1986.

\bibitem{G07} I. Gutman, B. Furtula, H. Hua, Bipartite unicyclic graphs with maximal, second
maximal, and third maximal energy, MATCH Commun. Math. Comput. Chem. 58
(2007) 85--92.

\bibitem{G78} I. Gutman, The energy of a graph, Ber. Math. Statist. Sekt. Forschungszenturm Graz. 103 (1978) 1--22.

\bibitem{G95} I. Gutman, S.L. Lee, J.H. Sheu, C. Li, Predicting the nodal properties of molecular orbitals by means of signed graphs, Bull. Inst. Chem. Acad. Sin. 42 (1995) 25--32.

\bibitem{GY01}I. Gutman, Y. Hou, Bipartite unicyclic graphs with greatest energy, MATCH Commun. Math. Comput. Chem. 43 (2001)
17--28.

\bibitem{P08} I. Pena, J. Rada, Energy of digraphs, Linear Multilinear Algebra 56 (5) (2008) 565--579.

\bibitem{P14} S. Pirzada, M. A. Bhat, Energy of signed digraphs, Discrete Appl. Math. 169 (2014) 195--205.

\bibitem{Z10} T. Zaslavsky, Matrices in the theory of signed simple graphs, in: Advances in
Discrete Mathematics and Applications, Mysore, 2008, in: Ramanujan Math. Soc.
Lect. Notes Ser., vol. 13, Ramanujan Math. Soc., 2010,pp. 207--229.





\end{thebibliography}
\end{document}